\documentclass[12pt]{article}
\usepackage{amsmath}
\usepackage{amsthm}
\usepackage{amssymb}
\usepackage{graphicx,psfrag,epsf}
\usepackage{enumerate}
\usepackage{natbib}
\usepackage{hyperref}
\usepackage{amsfonts}
\usepackage{bbm}
\usepackage{color} 
\usepackage[blocks]{authblk}
\usepackage{fullpage} 

\makeatletter
\newtheorem*{rep@theorem}{\rep@title}
\newcommand{\newreptheorem}[2]{%
\newenvironment{rep#1}[1]{%
 \def\rep@title{#2 \ref{##1}}%
 \begin{rep@theorem}}%
 {\end{rep@theorem}}}
\makeatother

\newtheorem{theorem}{Theorem}
\newreptheorem{theorem}{Theorem}
\newtheorem{lemma}{Lemma}
\newreptheorem{lemma}{Lemma}
\newtheorem{corollary}{Corollary}
\newreptheorem{corollary}{Corollary}

\begin{document}

\def\spacingset#1{\renewcommand{\baselinestretch}%
{#1}\small\normalsize} \spacingset{1}

%%%%%%%%%%%%%%%%%%%%%%%%%%%%%%%%%%%%%%%%%%%%%%%%%%%%%%%%%%%%%%%%%%%%%%%%%%%%%%

\title{Strong Consistency of Sparse K-means Clustering}
\date{} 
% first authors
\author[1]{JeungJu Kim}
\author[1]{Johan Lim}
% affiliation
\affil[1]{Department of Statistics, Seoul National University, Seoul, Korea} 

  \maketitle

\begin{abstract}
\noindent  In this paper, we study the strong consistency of the sparse K-means clustering for high dimensional data. 
We prove the consistency in both risk and clustering for the Euclidean distance. We discuss the characterization 
of the limit of the clustering under some special cases. For the general (non-Euclidean) distance, we prove the consistency in risk. 
Our result naturally extends to other models with the same objective function but 
different constraints such as $\ell_0$ or $\ell_1$ penalty in recent literature.

\medskip 
\noindent%
{\bf Keywords:}  empirical risk minimization; Euclidean distance; general distance; sparse K-means clustering; strong consistency. 
\end{abstract}

\baselineskip 18pt

%%%%%%%%%%%%%%%%%%%%%%%%%%%%%%%%%%%%%%%%%%%%%%%%%%%%%%%%%%%%%%%%
%%%%%%%%%%%%%%%%%%%%%%%%%%%%%%%%%%%%%%%%%%%%%%%%%%%%%%%%%%%%%%%%
%%%%%%%%%%%%%%%%%%%%%%%%%%%%%%%%%%%%%%%%%%%%%%%%%%%%%%%%%%%%%%%%
\section{Introduction}
\label{sec:intro}

 K-means clustering is a widely used method for clustering. However, in high-dimensional settings, the standard K-means procedure performs poorly 
 due to the presence of many irrelevant features. These features can obscure the true clusters by adding noise to the clustering process. To address this problem, 
 various techniques have been introduced to cluster high-dimensional data. 
 One such method, sparse K-means by \citet{witten:2010} has become a popular benchmark for high-dimensional clustering.

The sparse K-means clustering by \cite{witten:2010} selects features and performs clustering simultaneously. They formulated an optimization problem as 
    \begin{equation} \label{orig sparse kmeans}
    \begin{aligned}
        \max_{\mathbf{w}, C_1, \ldots, C_K} \quad & \sum_{j=1}^p w_j \biggl( \frac{1}{n} \sum_{i=1}^n \sum_{i'=1}^n d_{i,i',j}-\sum_{k=1}^K \frac{1}{n_k} \sum_{i, i' \in C_k} d_{i, i',j} \biggr) \\
        \textrm{s.t.} \quad & \|\mathbf{w}\|_2^2 \leq 1, \quad \|\mathbf{w}\|_1 \leq s, ~ \quad w_j \geq 0, \forall j
    \end{aligned},
    \end{equation}
 where $C_1, \ldots C_K$ are $K$ partitions of the data and $\mathbf{w} \in \mathbb{R}^p$ is a $p$-dimensional weight vector. The objective function is 
 a weighted between cluster sum of squares (BCSS) and $s$ is a tuning parameter to adjust the degree of sparsity. They then proposed a coordinate 
 descent algorithm which iteratively solves \eqref{orig sparse kmeans} by fixing partition and optimizing for the weight and vice versa. 
 The reason for 
 optimizing a weighted version of BCSS is quite intuitive. We may think of the weight as a coordinate-wise scale 
 transformation. Since different variables
  in the original data do not necessarily represent the right scale for clustering, it is reasonable to alter 
  the data to become more suitable for clustering.

Despite the success of the sparse K-means clustering, we know little about its theoretical properties. 
\citet{chakraborty:2020} suggested a strongly consistent lasso-weighted K-means clustering and a coordinate descent algorithm to solve it. 
While the consistency property of their estimator has been established by extending the work of \citet{pollard:1981}, 
the estimator requires three different hyperparameters, $\lambda, \alpha, \beta$, making it hard to 
implement and interpret its results. Moreover, the proof technique therein does not simply carry over to the sparse 
K-means method since the objective function of the latter is formulated in terms of the pairwise distance and is based on 
BCSS as opposed to the within-cluster sum of squares (WCSS) of the former.

In this paper, we aim to bridge this gap by showing the strong consistency of the center of sparse K-means when 
the distance is the squared Euclidean distance, which is commonly used in K-means clustering. i.e., $d_{i,i',j} = (X_{ij}-X_{i'j})^2$ in \eqref{orig sparse kmeans}. 
In addition, we show that the population version optimizer properly selects the relevant features by assuming a two-component uniform distribution. 
If non-Euclidean distance is used in clustering, the equivalence between centroid-based clustering (the clustering with WCSS) and partition-based clustering (the 
clustering with BCSS)  is 
not true anymore. However, for this case, we still show risk consistency results.

To prove the strong consistency of sparse K-means, we first alter the problem \eqref{orig sparse kmeans} into the centroid-based formulation. This equivalence was also utilized in the seminal paper by \citet{pollard:1981}, who showed the strong consistency of K-means clustering. Then, we cast this problem in the framework of an empirical risk minimization (ERM) problem, or equivalently M-estimation in the literature. 
Using empirical process theory, we prove the consistency in risk, and further prove its strong consistency by showing 
the continuity of the risk function. When non-Euclidean distance is used in sparse K-means clustering, 
the equivalence is no longer present and we have to deal with partitions itself instead of centroids. 
Still, by exploiting the concentration property of U-statistics in BCSS, 
which has been explored in \cite{clemencon:2014, li:2021}, we prove strong consistency in risk.

In the remainder of the paper, we state our main results in Section \ref{sec:main} and provide their proofs in Appendix. 
We conclude the paper in Section \ref{sec:discussion} with discussions on cluster consistency for 
the case of general distance. 

\section{Main results}\label{sec:main}

\subsection{Notations and Assumptions}      
   
 We denote by $\|\mathbf{x}\|_{\mathbf{w}}^2 = \sum_{j=1}^p w_j x_j^2$ for $\mathbf{x},\mathbf{w} \in \mathbb{R}^p$. 
    We call the following problem as the centroid-based formulation and it is of our main interest. 
    \begin{equation} \label{centroid sparse kmeans}
    \begin{aligned}
        \max_{\mathbf{w} \in \mathbb{R}^p, A \subset \mathbb{R}^p, \#A = K} \quad & \frac{1}{n} \sum_{i=1}^n (\|X_i-\bar{X}\|_\mathbf{w}^2 - \min_{a \in A} \|X_i-a\|_\mathbf{w}^2) \\
        \textrm{s.t.} \quad \qquad & \|\mathbf{w}\|_2^2 \leq 1, \quad \|\mathbf{w}\|_1 \leq s, \quad w_j \geq 0, \forall j
    \end{aligned}
    \end{equation}
    For the population version of the above formulation, we consider
    \begin{equation} \label{population sparse kmeans}
    \begin{aligned}
        \max_{\mathbf{w} \in \mathbb{R}^p, A \subset \mathbb{R}^p, \#A = K} \quad & \mathbb{E} \left[ \|X - \mu \|_\mathbf{w}^2 - \min_{a \in A} \|X-a\|_\mathbf{w}^2 \right] \\
        \textrm{s.t.} \quad \qquad & \|\mathbf{w}\|_2^2 \leq 1, \quad \|\mathbf{w}\|_1 \leq s, \quad w_j \geq 0, \forall j,
    \end{aligned}
    \end{equation}
    where $\mu = \mathbb{E}[X] \in \mathbb{R}^p$ is the mean of random vector $X$. We let the negated value of the objective of \eqref{population sparse kmeans} as $R(\mathbf{w},A)$, the clustering risk. The corresponding empirical risk is denoted by $R_n(\mathbf{w},A) = -\frac{1}{n} \sum_{i=1}^n (\|X_i-\mu\|_\mathbf{w}^2 - \min_{a \in A} \|X_i-a\|_\mathbf{w}^2)$. Note the slight difference of $\mu$ and $\bar{X}$ between $R_n$ and the objective function of \eqref{centroid sparse kmeans}. We denote by $R_n'(\mathbf{w},A)$ the negative value of the objective function of \eqref{centroid sparse kmeans}.

 Throughout the paper, we make use of two assumptions, which are presented below. 
 \begin{itemize} 
 \item[(A1)] $X$ has a compact support : $\exists M$ such that $|X_j| \le M$ a.s. for all $1 \le j \le p$.
 \item[(A2)] The optimal solution $\theta^* = (\mathbf{w}^*, A^*)$ of \eqref{population sparse kmeans} is unique. 
 \end{itemize} 
 We remark that the compact support assumption (A1) is quite common in the clustering and vector quantization literature \citep{bartlett:1998, levrard:2013, chakraborty:2023}. However, one should note that we are putting a milder condition than $||X|| \le M$ a.s. which more commonly appears in the literature. This assumption is crucial in our analysis since it facilitates the use of empirical process theory. Also, (A2) is important for proving the strong consistency since the convergence 
 notion may be obscure if there exists more than two minimizers. We remark that our risk consistency result, which may be of independent interest, 
 does not require (A2).

\subsection{Consistency for Euclidean distance} 

 We begin by establishing the equivalence between \eqref{orig sparse kmeans} and \eqref{centroid sparse kmeans} through the following lemma, where the proof is in the next section.
    \begin{lemma} \label{lem:1}
        The optimal values of \eqref{orig sparse kmeans} and \eqref{centroid sparse kmeans} are the same when $d_{i,i',j}=(X_{ij}-X_{i'j})^2$.
    \end{lemma}
    We denote by $\hat{\theta} = ({\bf \hat{w}}, \hat{A})$ the optimal solution of (\ref{centroid sparse kmeans}). We remark that the conversion from the solution 
    of \eqref{orig sparse kmeans} to that of \eqref{centroid sparse kmeans} is very straightforward; the latter naturally emerges during the process of running the algorithm, and the exact formula can be found in the proof. Now, we derive the risk consistency result.
    \begin{theorem} \label{thm:1}
        Under (A1), with probability at least $1-3t$,
	\[
	R(\hat{\theta}) - R(\theta^*) \le 4RC + 8sM^2\sqrt{\frac{2\log(1/t)}{n}} + 2sM^2 \frac{\log(p/t)}{n},
	\]
        where 
        \[
        RC \le \sqrt{\frac{2}{n}} sM^2 \left( \sqrt{K} + 5K \right)
        \]
    \end{theorem}
\noindent 
Applying the Borel-Cantelli lemma shows that, if $\log(p)/n \rightarrow 0$ as $n,p \rightarrow \infty$, $R(\hat{\theta}) - R(\theta^*) \rightarrow 0$ almost surely. 
For example, for $t$ that is summable to both $n$ and $p$ (for example, $t=(np)^{-2}$), the rate of the upper bound is 
\[
4 \sqrt{\frac{2}{n}} sM^2 \left( \sqrt{K} + 5K \right) + 8sM^2\sqrt{\frac{2\log(1/t)}{n}} + 2sM^2 \frac{\log(p/t)}{n}  \sim s \max\left( \frac{\log p}{n}, \sqrt{\frac{-\log t}{n}}\right),
\]
with the assumption of fixed $K$ (if $t=(np)^{-2}$, the rate is $s \log p /n$). This means that as long as the dimension does not grow exponentially, the risk consistency result is obtained. In fact, if we put a stronger assumption on the support of $X$ such as $\exists M \text{ such that } ||X|| \le M \text{ } a.s.$, then the dimension-free risk result follows. This is in line with preexisting literature on K-means clustering by \citet{Biau:2008} where they also exploited a stronger assumption than coordinate-wise compactness.
This can be interpreted to mean that as long as the solution to the empirical risk function is found, the dimension plays little role on its performance. However, one should not be misled to believe that dimensionality does not play any role in sparse K-means clustering, as optimization usually becomes harder as dimension increases.	
    
  To derive strong consistency from the risk consistency, one may be interested in finding a sufficient condition for $R(\hat{\theta}) \rightarrow R(\theta)$ implies $\hat{\theta} \rightarrow \theta$. The condition below guarantees such property.
    \begin{equation} \label{cond:1}
    \forall \epsilon >0, \exists \eta > 0 \quad \mathrm{s.t.} \quad d(\hat{\theta}, \theta^*) \ge \epsilon \implies R(\hat{\theta}) \ge R(\theta^*) + \eta
    \end{equation}
    Under (A1) and (A2), the continuity of $\theta \mapsto R(\theta)$ is sufficient for (\ref{cond:1}) simply by taking 
    \[
    \eta = \min_{\theta : d(\theta, \theta^*) \ge \epsilon}R(\theta) - R(\theta^*),
    \]
 where the minimum is attained by the extreme value theorem \citep{rudin:1964} and $\eta > 0$, which states 
 the uniqueness of the minimizer, follows from (A2). Our next theorem states the continuity of the risk function.
    \begin{theorem} \label{thm:2}
        The map
        \[
            ({\bf w},A) \rightarrow R({\bf w}, A) = \mathbb{E} \left[ \|X - \mu\|_\mathbf{w}^2 - \min_{a \in A} \|X-a\|_\mathbf{w}^2 \right]
        \]
        is continuous, where $d(({\bf w}_1,A_1), ({\bf w}_2,A_2)) = \max \{ \|{\bf w_1}-{\bf w_2}\|, d_H(A_1, A_2)\}$, and $d_H$ denotes Hausdorff metric between two sets.
    \end{theorem}
\noindent The main idea of the proof is from \citet{evans:2024} and can be found in Section \ref{sec:proofs}. 
We remark that this continuity property requires neither (A1) nor (A2). 
Now that the continuity result is established, $\hat{\theta} \rightarrow \theta^*$ a.s. follows as a corollary. 
\begin{corollary} \label{cor:1}
	Under (A1) and (A2), $\hat{\theta} \rightarrow \theta^*$ a.s. as $n,p \rightarrow \infty$
\end{corollary} 
\begin{proof}
The proof is immediate from Theorem \ref{thm:1} and Theorem \ref{thm:2}. 
\end{proof}
 
 \subsection{Consistency for general distance}
	
If data are not generated from Euclidean space anymore, we can no longer reformulate \eqref{orig sparse kmeans} into  
\eqref{centroid sparse kmeans}. Thus, we have to deal with random partition $C_1, \ldots, C_K$ itself instead of more tractable $K$ 
points $a_1, \ldots, a_K$. In this section, we prove a risk consistency result for the sparse K-means clustering for this 
general distance case.

First, we define 
\[
\Pi = \big\{\{C_1, \ldots, C_K\} : \cup_{i=1}^K C_i = \mathcal{X}, C_i \cap C_j = \emptyset, \forall i \ne j  \big\},
\] 
a collection of $K$-partitions whose union forms $\mathcal{X}$. Further regularity condition shall be put on $\Pi$. 
\begin{itemize} 
\item[(A3)] $\exists \delta>0 \text{ s.t. } \forall \{C_1, \ldots, C_K\} \in \Pi, \min_{1 \le i \le K}P(C_i) \ge \delta$
\end{itemize}
This assumption is, in general, hard to verify empirically since we do not have information about the probability measure. 
We remark that this can be replaced by more general assumptions such as 
\begin{itemize} 
\item[(A4)] $X_1, \ldots, X_n$ are continuously distributed with pdf $f$ and on compact support $\mathcal{X}$, $f>0$.
\item[(A5)] $\exists \delta>0 \text{ s.t. } \forall \{C_1, \ldots, C_K\} \in \Pi, \min_{1 \le i \le K}\text{vol}(C_i) \ge \delta$. 
\end{itemize}
Also, since there is no notion of coordinates in general metric space, we relax the coordinate-wise compact support assumption (A1) as
\begin{itemize}  
\item[($\mbox{A1}^{\prime}$)] The diameter of $\mathcal{X}$ is bounded by $M<\infty$. 
\end{itemize} 
Lastly, the population problem is defined as
    \begin{equation} \label{general population sparse kmeans}
    \begin{aligned}
        \max_{\mathbf{w}, C_1, \ldots, C_K} \quad & \sum_{j=1}^p w_j \left( \mathbb{E}d_j(X_1, X_2) -\sum_{k=1}^K \frac{1}{P(C_k)} \mathbb{E}\left[d_j(X_1, X_2)I\{(X_1, X_2) \in C_k^2\} \right] \right) \\
        \textrm{s.t.} \quad & \|\mathbf{w}\|_2^2 \leq 1, \quad \|\mathbf{w}\|_1 \leq s, \quad w_j \geq 0, \forall j,
    \end{aligned}
    \end{equation}
  and as before, the objective function is denoted by $R(\mathbf{w}, (C_1, \ldots, C_K))$, the risk.

Under these conditions, the following theorem is derived.
\begin{theorem}\label{thm:4}
	Let $\hat{\theta}= (\hat{\bf{w}}, \{\hat{C_1}, \ldots, \hat{C_K}\})$ denote the minimizer of (\ref{orig sparse kmeans}) over $\mathcal{F} \times \Pi$, where $\mathcal{F} = \{ \mathbf{w} \in \mathbb{R}^p : ||\mathbf{w}||_2^2 \leq 1, ||\mathbf{w}||_1 \leq s, w_j \geq 0, \forall j \}$. Also denote by $\theta^* = (\mathbf{w}^*, \{C_1^*, \ldots, C_K^*\})$ the minimizer of corresponding population problem (\ref{general population sparse kmeans}) over $\mathcal{F} \times \Pi$. Assume ($\mbox{A1}^{\prime}$) and (A3). Then, with probability at least $1-4pt$, 
\begin{eqnarray} 
	R(\hat{\theta}) - R(\theta^*) &\le&  2sM \sqrt{\frac{2}{n}\log(1/t)} + \frac{4sKM}{\delta^2} \left(2RC+ \sqrt{\frac{2}{n}\log(1/t)}\right) \nonumber\\
	&& \qquad  \qquad +  \frac{2sK}{\delta}\left( 2 \max_{1 \le j \le p}RC_j + M\sqrt{\frac{2}{n}\log(1/t)} \right) \nonumber 
\end{eqnarray} 
provided that $2RC+\sqrt{\frac{2}{n}\log{1/t}} \le \frac{\delta}{2}$, where
\begin{eqnarray}
	RC &=& \mathbb{E}\sup_{C \in \mathcal{P}, \mathcal{P} \in \Pi}\frac{1}{n} \left|\sum_{i=1}^n \epsilon_i \mathbb{I}(X_i \in C) \right| \nonumber\\
	RC_j &=& \mathbb{E}\sup_{C \in \mathcal{P}, \mathcal{P} \in \Pi}\frac{1}{\lfloor n/2 \rfloor} \left| \sum_{i=1}^{\lfloor n/2 \rfloor} \epsilon_i d_j(X_i, X_{i+\lfloor n/2 \rfloor}) \mathbb{I}((X_i,X_{i+\lfloor n/2 \rfloor}) \in C^2) \right|. \nonumber
	\end{eqnarray} 
\end{theorem}

Naturally, in our analysis, the concentration of U-statistics is taken into account due to the BCSS part in our clustering criterion. This idea of applying U-process theories into clustering is well explored in \cite{clemencon:2014}, and our proof rests on it.

We make a few remarks on this theorem. First, the result of this theorem can be restated as (omitting constants in the $n, p \rightarrow \infty$ regime),
\[
	R(\hat{\theta}) - R(\theta^*) \lesssim \sqrt{\frac{1}{n}\log(p/t)} + RC + \max_{1 \le j \le p} RC_j
\]
with probability at least $1-t$. Therefore, in the $(n,p)$ regime where
\[
	\frac{\log{p}}{n}, RC, \max_{1 \le j \le p}RC_j \rightarrow 0, \text{ as } n, p \rightarrow \infty,
\]
the risk consistency holds true by the first Borel-Cantelli lemma. Since it is, in general, hard to directly evaluate $RC$, we present a simple corollary that links the concept of the Vapnik-Chervonenkis (VC) dimension to the Rademacher complexity.
\begin{corollary}\label{cor:2}
	Suppose the VC dimension of $\mathcal{A} = \{C \subset \mathcal{X} \mid C \in \mathcal{P}, \mathcal{P} \in \Pi \}$ is $v$, which may depend on $p$. Then, $RC$ and $\max_{1 \le j \le p} RC_j$ are of order $O(\sqrt{\frac{v}{n}})$. Consequently, the solution to \eqref{orig sparse kmeans} is risk consistent as long as $\frac{\max(\log p, v)}{n} \rightarrow 0$ as $n,p \rightarrow \infty$.
\end{corollary}

\begin{proof}
Here, we present an outline of the proof as bounding the Rademacher Complexity using VC dimension is quite a standard technique (for example, see Example 5.24 of \citet{wainwright:2019}). 
Note that for each $j$, $\mathcal{F}_j = \{d_j(\cdot, \cdot)\mathbb{I}((\cdot,\cdot) \in C^2): C \in \mathcal{P}, \mathcal{P} \in \Pi \}$ has the same VC subgraph dimension as the VC dimension of $\mathcal{A}$. These classes all share the same envelope function $d(\cdot, \cdot)$ and the covering number is bounded by 
\[
	N(\epsilon; \mathcal{F}_j, ||\cdot||_{\mathbb{P}_n}) \le \left( \frac{c_1}{\epsilon} \right)^{c_2v}
\]
for some universal constants $c_1, c_2 > 0$ that are independent of $j$. Finally, plugging this estimate into the following Dudley's entropy integral gives the claim. 
\[
	RC_j \lesssim \frac{1}{\sqrt{n}} \int_0^{2M} \sqrt{\log N(t;\mathcal{F}_j;||\cdot||_{\mathbb{P}_n})} dt
\]
The case for $RC$ follows in a similar way.
\end{proof}

In the particular scenario where the underlying data space is the Euclidean space and $\Pi$ is the collection of Voronoi partitions with respect to the Euclidean norm, $v = O(p)$ (Theorem 21.5 of \citet{devroye:2013}) and risk consistency holds as long as $p/n \rightarrow 0$ as $n,p \rightarrow \infty$. This partly recovers the result presented in our previous Theorem \ref{thm:1}. Nonetheless, we acknowledge that there remains a slight gap between both results as Theorem \ref{thm:1} puts a milder restriction on the order of $p$.

Finally, we remark that our analysis takes into account the normalization part $1/n_k$ present in the BCSS clustering criterion, which is essential to establish the equivalence between centroid-based clustering and partition-based clustering. This is in contrast to many of the current analyses of partition-based clustering performances \citep{clemencon:2014,li:2021}, where their frameworks do not consider the normalized objective function. 

\section{Discussion} \label{sec:discussion}

We conclude the paper with two further discussions on the results of the paper, the characterization of the cluster limit 
for Euclidean distance and the consistency in clustering for general distance.

First, let us discuss the characterization of the cluster limit of the Euclidean case, in which we are able to prove the strong consistency 
of the cluster. However, even for this case, the limit process \eqref{population sparse kmeans}
 is too complicated to directly analyze. We try to characterize the limit by assuming a two-component uniform mixture model and try to figure out if \eqref{population sparse kmeans} correctly recovers the weight and clusters. We consider a uniform distribution on the union of two balls $\bigcup_{i=1}^2 B(a_i, \sqrt{r})$, where $a_1 = (0, \ldots, 0)^t$ and $a_2 = (\overbrace{1, \ldots, 1}^{r}, 0, \ldots, 0)^t$. For this model, we prove the following theorem.
 
    \begin{theorem} \label{thm:3}
        Let $X$ be a random vector taking values in $\mathbb{R}^p$ that follows a uniform distribution on $\bigcup_{i=1}^2 B(a_i, \sqrt{r}/2)$, where $a_1 = (0, \ldots, 0)^t$ and $a_2 = (\overbrace{1, \ldots, 1}^{r}, 0, \ldots, 0)^t$. Then, $\mathbf{w} = (\overbrace{1, \ldots, 1}^{r}, 0, \ldots, 0)^t$ and $A = \{a_1, a_2\}$ is a stationary point to \eqref{population sparse kmeans}.
    \end{theorem}
    
We remark that since our proof technique greatly rests on the symmetry argument, it is not straightforward to extend this result to the case, where $a_2$ doesn't have the same value for the first $r$ components. 
Also, the fact that two components of uniform distribution do not share the supports plays a crucial role. In fact, if we consider the two-component normal mixture model, this conclusion no longer holds. Consider the Gaussian mixture model, $X \sim \frac{1}{2}N_p(\mu_1,\sigma^2 I_p) + \frac{1}{2}N_p(\mu_2,\sigma^2 I_p)$, where $\mu_1 = (0, \ldots, 0)^t$ and $\mu_2 = (\overbrace{\delta, \ldots, \delta}^{r}, 0, \ldots, 0)^t$ for some $\delta>0$. Given $\mathbf{w} = (\overbrace{\alpha, \ldots, \alpha}^{r}, 0, \ldots, 0)^t$, $\alpha>0$, we cannot recover $A = \{\mu_1, \mu_2\}$. This fact follows from the necessary condition of optimal quantizer \citep[Theorem 4.1]{graf:2007} as the mean of truncated normal distribution is no longer the same as $\mu_1$.
	
Second, one might question whether the consistency of clusters could be derived from risk consistency for general distances, similar to the approach used for Euclidean distance. Developing this idea requires a proper mathematical framework for partition spaces, a set of every possible partition, as well as establishing appropriate notions of distance and compactness within these spaces. However, we defer this exploration to future work.

\section{Appendix: Proofs}
\label{sec:proofs}

\subsection{Proofs of Euclidean distance} 

\medskip 
In this section, we prove the theorems and lemmas stated above.
\begin{replemma}{lem:1}
    The optimal values of (\ref{orig sparse kmeans}) and (\ref{centroid sparse kmeans}) are the same when $d_{i,i',j}=(X_{ij}-X_{i'j})^2$.
\end{replemma}
\begin{proof}
    We begin by reformulating the (\ref{orig sparse kmeans}) by specifying that the distance used is the squared Euclidean distance, $d_{i,i',j}=(X_{ij}-X_{i'j})^2$. Then the problem becomes equivalent to maximizing
\begin{equation} \label{equi1}
    \sum_{i=1}^n ||X_i-\bar{X}||_{\mathbf{w}}^2 - \sum_{k=1}^K \sum_{i \in C_k} ||X_i-\bar{X}_k||_{\mathbf{w}}^2,
\end{equation}
where $\bar{X}_k = |C_k|^{-1} \sum_{i \in C_k} X_i$.
This follows from
\begin{equation}
\begin{aligned}
    \sum_{j=1}^p w_j \sum_{i=1}^n \sum_{i'=1}^n d_{i,i',j} 
    &= \sum_{j=1}^p w_j \sum_{i=1}^n \sum_{i'=1}^n (X_{ij}-X_{i'j})^2 \\
    &= \sum_{j=1}^p w_j \sum_{i=1}^n \sum_{i'=1}^n (X_{ij}-\bar{X}_{\cdot j} + \bar{X}_{\cdot j} + X_{i'j})^2 \\
    &= \sum_{j=1}^p w_j \sum_{i=1}^n \sum_{i'=1}^n (X_{ij}-\bar{X}_{\cdot j})^2 + (\bar{X}_{\cdot j} + X_{i'j})^2 \\
    &= 2n \sum_{i=1}^n ||X_i-\bar{X}||_{\mathbf{w}}^2.
\end{aligned}
\end{equation}
Note that the decision variables of the objective (\ref{equi1}) are partitions and weight. We further claim that maximizing (\ref{equi1}) is equivalent to the maximization of (\ref{centroid sparse kmeans}), which proves the lemma.
For every feasible solution $(\mathbf{w},C_1, \ldots, C_K)$ of (\ref{orig sparse kmeans}), let $a_1=\bar{X}_1, \cdots, a_K=\bar{X}_K$, where $\bar{X}_i$ denotes the mean vector of $C_i$. 
\begin{equation*}
    \begin{aligned}
        &\sum_{i=1}^n \left(||X_i-\bar{X}||_\mathbf{w}^2 - \min_{a \in A} ||X_i-\theta||_\mathbf{w}^2 \right)\\ 
        &= \sum_{i=1}^n \left(||X_i-\bar{X}||_\mathbf{w}^2 - \min_{1 \leq j \leq K} ||X_i-\bar{X_j}||_\mathbf{w}^2 \right) \\
        &= \sum_{i=1}^n ||X_i-\bar{X}||_\mathbf{w}^2 - \sum_{k=1}^K \sum_{i \in C_k} \min_{1 \leq j \leq K} ||X_i-\bar{X_j}||_\mathbf{w}^2 \\
        &\geq \sum_{i=1}^n ||X_i-\bar{X}||_\mathbf{w}^2 - \sum_{k=1}^K \sum_{i \in C_k} ||X_i-\bar{X_k}||_\mathbf{w}^2.
    \end{aligned}
\end{equation*}
Conversely, for every feasible solution $(\mathbf{w}, a_1, \ldots, a_K)$ of (\ref{centroid sparse kmeans}), let $C_i=\{X_l:||X_l-a_i||_{\mathbf{w}} = \min_{1 \leq j \leq K}||X_l-a_j||_{\mathbf{w}}, l \in \{1,\ldots,n\}\}, \forall i \in \{1,\ldots,K\}$.
\begin{equation*}
    \begin{aligned}
        &\sum_{i=1}^n ||X_i-\bar{X}||_\mathbf{w}^2 - \sum_{k=1}^K \sum_{i \in C_k} ||X_i-\bar{X_k}||_\mathbf{w}^2\\
        &\geq \sum_{i=1}^n ||X_i-\bar{X}||_\mathbf{w}^2 - \sum_{k=1}^K \sum_{i \in C_k} ||X_i-a_k||_\mathbf{w}^2 \\
        &= \sum_{i=1}^n ||X_i-\bar{X}||_\mathbf{w}^2 - \sum_{i=1}^n \min_{1 \leq j \leq K} ||X_i-a_j||_\mathbf{w}^2
    \end{aligned}
\end{equation*}
\end{proof}

\begin{reptheorem}{thm:1}
        Under (A1), with probability at least $1-3t$,
	\[
	R(\hat{\theta}) - R(\theta^*) \le 4RC + 8sM^2\sqrt{\frac{2\log(1/t)}{n}} + 2sM^2 \frac{\log(p/t)}{n},
	\]
        where 
        \[
        RC \le \sqrt{\frac{2}{n}} sM^2 \left( \sqrt{K} + 5K \right)
        \]
\end{reptheorem}
\begin{proof}
    The result rests on the classical inequality
    \begin{equation}
    \begin{aligned}
        R(\hat{\theta}) - R(\theta^*) &\leq \sup_{\theta} (R_n(\theta) - R(\theta)) + \sup_{\theta} (R(\theta) - R_n(\theta)) + 2\sup_{\theta} |R_n(\theta) - R_n'(\theta)|  \\
        &= \sup_{\theta} (R_n(\theta) - R(\theta)) + \sup_{\theta} (R(\theta) - R_n(\theta)) + 2 ||\bar{X}-\mu||_{\bf{w}}^2
    \end{aligned}
    \end{equation}
    and bounding the empirical process $\sup_{\theta} (R_n(\theta) - R(\theta))$ via Rademacher complexity.
    \begin{equation}
        RC = \mathbb{E}  \left[ \sup_{f \in \mathcal{F}} \frac{1}{n} \sum_{i=1}^n \epsilon_i f(x_i) \right],
    \end{equation}
    where $\mathcal{F} = \{||\cdot-\mu||_{\bf{w}}^2 - \min_{a \in A} ||\cdot-a||_{\bf{w}}^2 : ||{\bf w}||_1 \leq s, ||{\bf w}||_2^2 \leq 1, A \subset \mathbb{R}^p, |A|=k \}$ and $\epsilon_i$ are idependent and identically distributed Rademacher variables. \\
    We shall apply the vector contraction theorem from \cite{maurer:2016} to show that 
    \[
    RC \lesssim \frac{1}{\sqrt{n}}.
    \]
	First, note that $||\cdot-\mu||_{\bf{w}}^2 - \min_{a \in A} ||\cdot-a||_{\bf{w}}^2 = \max_{a \in A} \{ ||\cdot-\mu||_{\bf{w}}^2 - ||\cdot-a||_{\bf{w}}^2 \}$.
        Since $(b_1, \cdots, b_K) \mapsto \max \{b_1, \cdots, b_K\}$, for $b_i \in \mathbb{R}$ is a 1-Lipschitz function with respect to the Euclidean distance, we can apply the vector-contraction inequality from \citet{maurer:2016}. 
    \begin{equation}
        \begin{aligned}
            n \times RC &\leq \sqrt{2} \mathbb{E} \sup_{w,A} \sum_{i=1}^n \sum_{k=1}^K \epsilon_{ik} \left( ||x_i-\mu||_w^2 - ||x_i-a_k||_w^2 \right) \\
            &= \sqrt{2} \mathbb{E} \sup_{w,A} \sum_{i=1}^n \sum_{k=1}^K \epsilon_{ik} \left( ||\mu||_w^2 - ||a_k||_w^2 -2\langle x_i, \mu-a_k \rangle_w \right) \\
            &\le \sqrt{2} \mathbb{E} \left[ \sup_{w,A} \sum_{i,k} \epsilon_{ik} ||\mu||_w^2 + \sup_{w,A} \sum_{i,k} \epsilon_{ik} ||a_k||_w^2 + \sup_{w,A} \sum_{i,k} 2 \epsilon_{ik} \langle x_i, \mu-a_k \rangle_w \right]\\
            &\overset{(i)}{\le} \sqrt{2} \left[sM^2\mathbb{E} \left|\sum_{i,k}\epsilon_{ik} \right| + sKM^2 \mathbb{E} \left|\sum_{i}\epsilon_{i} \right| + 4\sqrt{s}KM \mathbb{E} \left\|\sum_{i}\epsilon_{i} \sqrt{w} \odot X_i \right\|\right] \\
            &\overset{(ii)}{\le} \sqrt{2} \left( sM^2 \sqrt{nK} + sKM^2 \sqrt{n} + 4sKM\sqrt{nM^2} \right)
        \end{aligned}
    \end{equation}
    For $(i)$,
    \begin{equation}
        \begin{aligned}
            \mathbb{E} \sup_{w,A} \sum_{i,k} \epsilon_{ik} \langle x_i, \mu-a_k \rangle_w  &= \mathbb{E} \sup_{w,A} \sum_{k}  
            \left\langle \sum_{i} \sqrt{w} \odot \epsilon_{ik} x_i, \sqrt{w} \odot (\mu-a_k) \right \rangle \\
            &\le \sum_{k} \mathbb{E} \left[ \sup_{w,A} \|\sum_{i} \sqrt{w} \odot \epsilon_{ik} x_i \| \|\sqrt{w} \odot (\mu-a_k)\| \right] \\
            &\le 2KM \mathbb{E} \left\| \sum_{i}\epsilon_{i} \sqrt{w} \odot X_i \right\|,
        \end{aligned}
    \end{equation}
    where $\epsilon_i, \epsilon_{ik}$ are iid rademacher variables, $a_k$ are the elements of $A$, $\sqrt{w} \in \mathbb{R}^p$ is the square root applied to each element of $w$ and $\odot$ refers to the elementwise multiplication. 
    The inequality $(ii)$ follows from Jensen's inequality. This proves that $RC = O(\frac{1}{\sqrt{n}})$.
    This implies that the rate of the empirical process is $O(\frac{1}{\sqrt{n}})$. To be precise, with probability at least $1-t$,
    \[
    \sup_\theta (R_n(\theta) - R(\theta)) \leq 2 RC + 4 s M^2\sqrt{\frac{2\log(1/t)}{n}} = O(\frac{1}{\sqrt{n}}),
    \]
    which follows from bounded difference inequality together with standard symmetrization argument and noting that our function 
    class $\mathcal{F}$ is uniformly bounded by $4sM^2$ (for example, see Theorem 4.10 in \citet{wainwright:2019}). Similarly, with probability at least $1-t$,
    \[
    \sup_\theta (R(\theta) - R_n(\theta)) \leq 2 RC + 4 s M^2\sqrt{\frac{2\log(1/t)}{n}} = O(\frac{1}{\sqrt{n}}).
    \]
    Lastly, we bound $||\bar{X}-\mu||_{\mathbf{w}}^2$ using the Hoeffding inequality together with the union bound. As a result, it follows that with probability at least $1-t$,
\[
	||\bar{X}-\mu||_{\mathbf{w}}^2 \le \sum_{j=1}^p w_j \frac{2\log(p/t)}{n}M^2 \le 2sM^2 \frac{\log(p/t)}{n}
\]
    Putting these all together proves the theorem.
\end{proof}

\begin{reptheorem}{thm:2}
    The map
    \[
        (A, {\bf w}) \rightarrow R(A, {\bf w}) = \mathbb{E} \left[ ||X - \mu||_\mathbf{w}^2 - \min_{a \in A} ||X-a||_\mathbf{w}^2 \right]
    \]
    is continuous, where $d((A_1, {\bf w_1}), (A_2, {\bf w_2})) = \max \{ d_H(A_1, A_2), ||{\bf w_1}-{\bf w_2}|| \}$, and $d_H$ denotes Hausdorff metric between two sets.
\end{reptheorem}
\begin{proof}
        We first start by proving ``Peter-Paul" inequality.
    \begin{lemma} \label{peter paul}
        $\forall \epsilon>0, \exists c_\epsilon>0$ such that $d^2(\mathbf{x},\mathbf{y}) \leq (1+\epsilon)d^2(\mathbf{x},\mathbf{z}) + c_\epsilon d^2(\mathbf{z},\mathbf{y})$ for every metric $d$, and $\mathbf{x},\mathbf{y},\mathbf{z} \in \mathbb{R}^p$.
    \end{lemma}
    \begin{proof}
    \begin{equation*}
    \begin{aligned}
        d^2(\mathbf{x}, \mathbf{y}) &\overset{(i)}{\le} \{d(\mathbf{x}, \mathbf{z})+ d(\mathbf{z}, \mathbf{y})\}^2 \\
        &= \left\{ \frac{1}{1+\epsilon} (1+\epsilon) d(\mathbf{x}, \mathbf{z}) + \frac{\epsilon}{1+\epsilon} \frac{1+\epsilon}{\epsilon}  d(\mathbf{z}, \mathbf{y}) \right\}^2 \\
        &\overset{(ii)}{\le} \frac{1}{1+\epsilon} (1+\epsilon)^2 d^2(\mathbf{x}, \mathbf{z}) + \frac{\epsilon}{1+\epsilon} \left(\frac{1+\epsilon}{\epsilon}\right)^2 d^2(\mathbf{z}, \mathbf{y}) \\
        &= (1+\epsilon) d^2(\mathbf{x}, \mathbf{z}) + \left(1+\frac{1}{\epsilon}\right) d^2(\mathbf{z}, \mathbf{y})
    \end{aligned}
    \end{equation*}
    The $(i)$ holds by triangle inequality and $(ii)$ by the convexity of $x^2$. Letting $c_\epsilon = 1+\frac{1}{\epsilon}$ proves the lemma.
\end{proof}
We can extend the above lemma to the distance between a set and a point and the distance between two sets, which is the Hausdorff distance. For this, let us define the necessary concepts.
    $$d(\mathbf{x},A) = \inf_{\mathbf{y} \in A} d(\mathbf{x}, \mathbf{y})$$
    $$\overrightarrow{d_H}(A,B) = \sup_{\mathbf{x} \in A}d(\mathbf{x}, B)$$
    $$d_H(A,B) = \max\left\{\overrightarrow{d_H}(A,B), \overrightarrow{d_H}(B,A)\right\}$$
Note that while $d$ and $d_H$ are metrics, but $\overrightarrow{d_H}$ is not a metric because it is not symmetric in general.

\begin{lemma} \label{peter paul2}
    $\forall \epsilon>0, \exists c_\epsilon >0$ such that $d^2(\mathbf{x},A) \leq (1+\epsilon)d^2(\mathbf{x},B) + c_\epsilon \overrightarrow{d_H}^2(B,A)$ for every metric d, $\mathbf{x} \in \mathbb{R}^p$, and $A,B \subset \mathbb{R}^p$.
\end{lemma}
\begin{proof}
    By Lemma \ref{peter paul}, $\exists c_\epsilon$ such that $\forall \mathbf{y} \in A, \forall \mathbf{z} \in B$, 
    \[
        d^2(\mathbf{x}, \mathbf{y}) \leq (1+\epsilon)d^2(\mathbf{x}, \mathbf{z}) + c_\epsilon d^2(\mathbf{z}, \mathbf{y}).
    \]
    Taking infimum over $\mathbf{y} \in A$ yields
    \[
        d^2(\mathbf{x}, A) \leq (1+\epsilon)d^2(\mathbf{x}, \mathbf{z}) + c_\epsilon d^2(\mathbf{z}, A).
    \]
    Finally, taking infimum again for $\mathbf{z} \in B$ gives 
    \begin{equation*}
    \begin{aligned}
        d^2(\mathbf{x}, A) &\le \inf_{\mathbf{z} \in B} \left\{(1+\epsilon)d^2(\mathbf{x}, \mathbf{z}) + c_\epsilon d^2(\mathbf{z}, A) \right\} \\
        &\le \inf_{\mathbf{z} \in B} \left\{(1+\epsilon)d^2(\mathbf{x}, \mathbf{z}) + c_\epsilon \sup_{\mathbf{z} \in B} d^2(\mathbf{z}, A) \right\} \\
        &= (1+\epsilon)d^2(\mathbf{x}, B) + c_\epsilon \overrightarrow{d_H}^2(B,A)
    \end{aligned}
    \end{equation*}
    and this completes the proof.
\end{proof}

\begin{lemma} \label{continuity of distance}
    $d_{\mathbf{w}_n}(x,y) \rightarrow d_{\mathbf{w}}(x,y)$ as $\mathbf{w}_n \rightarrow \mathbf{w}$, for every $x, y \in \mathbb{R}^p$. Furthermore, $d_{\mathbf{w}_n}(x,A) \rightarrow d_{\mathbf{w}}(x,A)$ as $\mathbf{w}_n \rightarrow \mathbf{w}$, for every $x \in \mathbb{R}^p$ and $A \subset \mathbb{R}^p$ such that $|A|<\infty$.
\end{lemma}
\begin{proof}
    The first assertion is immediate from its definition. For the second one, note that the finiteness of $|A|$ implies
    \begin{equation*}
        \max_{a \in A} \left|d_{\mathbf{w}_n}(x,a) - d_\mathbf{w}(x,a)\right| \rightarrow 0.
    \end{equation*}
    Then,
    \begin{equation*}
    \begin{aligned}
        \min_{a \in A} d_\mathbf{w}(x,a) &= \min_{a \in A} \left\{ d_\mathbf{w}(x,a) - d_{\mathbf{w}_n}(x,a) + d_{\mathbf{w}_n}(x,a) \right\} \\
        &\le \min_{a \in A} d_{\mathbf{w}_n}(x,a) + \max_{a \in A} \left\{ d_\mathbf{w}(x,a) - d_{\mathbf{w}_n}(x,a) \right\} \\
        &\le \min_{a \in A} d_{\mathbf{w}_n}(x,a) + \max_{a \in A} \left| d_\mathbf{w}(x,a) - d_{\mathbf{w}_n}(x,a) \right|.
    \end{aligned}
    \end{equation*}
    With the role of $d_\mathbf{w}$ and $d_{\mathbf{w}_n}$ reversed,
    \[
        \left| \min_{a \in A} d_\mathbf{w}(x,a) - \min_{a \in A} d_{\mathbf{w}_n}(x,a) \right| \le \max_{a \in A} \left|d_{\mathbf{w}_n}(x,a) - d_\mathbf{w}(x,a)\right|,
    \]
    and this completes the proof.
\end{proof}
Now, we are ready to prove the continuity.
    Suppose $(A_n, \mathbf{w}_n) \rightarrow (A, \mathbf{w})$ in $d_H \times d$ where $d_H$ denotes Hausdorff distance and $d$ denotes standard $p$-dimensional Euclidean distance. Our goal is to show that $R(A_n, \mathbf{w}_n) \rightarrow R(A, \mathbf{w})$.
    \begin{equation*}
    \begin{aligned}
        R(A, \mathbf{w}) &= \int \left\{ d^2_\mathbf{w}(x,\mu) - d^2_\mathbf{w}(x,A) \right\} d\mathbb{P}(x)\\
        &= \sum_{l=1}^p w_l Var(X_l) - \int d^2_\mathbf{w}(x,A)d\mathbb{P}(x)    
    \end{aligned}
    \end{equation*}
    Since $\mathbf{w}_n \rightarrow \mathbf{w}$ in $d$, it is clear that the first term of $R(A_n, \mathbf{w}_n)$ converges to that of $R(A, \mathbf{w})$. Thus, it remains to show that $\int d_{\mathbf{w}_n}^2(x,A_n) d\mathbb{P}(x) \rightarrow \int d_\mathbf{w}^2(x,A) d\mathbb{P}(x)$ as $n \rightarrow \infty$.

    For every $\epsilon>0$, pick $c_{\epsilon}$ in Lemma \ref{peter paul2} such that 
    \begin{equation} \label{peter paul3} 
        d^2(x, A) \leq (1+\epsilon) d^2(x, B) + c_\epsilon \overrightarrow{d_H}^2(B, A).
    \end{equation}
    Now, let $d, A, B$ be $d_{\mathbf{w}_n}, A_n, A$ respectively and then integrate with respect to the measure $\mathbb{P}$ which yields
    \begin{equation} \label{peter paul int}
        \int d_{\mathbf{w}_n}^2(x, A_n) d\mathbb{P}(x) \le (1+\epsilon) \int d_{\mathbf{w}_n}^2(x, A) d\mathbb{P}(x) + c_\epsilon \overrightarrow{d_{H, \mathbf{w}_n}}^2(A, A_n).
    \end{equation}
    As $n \rightarrow \infty$, $\overrightarrow{d_{H, \mathbf{w}_n}}^2(A, A_n) \rightarrow 0$ because
    \[
    \overrightarrow{d_{H, \mathbf{w}_n}}^2(A, A_n) \le  \overrightarrow{d_{H, s \mathbbm{1}}}^2(A, A_n) = s^2 \overrightarrow{d_H}^2(A, A_n) \rightarrow 0.
    \]
    Taking $\limsup_{n \rightarrow \infty}$ at \eqref{peter paul int}, one gets
    \begin{equation} \label{limsup cont}
    \begin{aligned} 
        \limsup_{n \rightarrow \infty} \int d_{\mathbf{w}_n}^2(x, A_n) d\mathbb{P}(x)&\le (1+\epsilon) \limsup_{n \rightarrow \infty} \int d_{\mathbf{w}_n}^2(x, A) d\mathbb{P}(x) \\
        &\le (1+\epsilon) \int \limsup_{n \rightarrow \infty} d_{\mathbf{w}_n}^2(x,A) d\mathbb{P}(x) \\
        &= (1+\epsilon) \int d_\mathbf{w}^2(x, A) d\mathbb{P}(x),
    \end{aligned}
    \end{equation}
    where we used the reverse Fatou's lemma for the second inequality. To check the condition for the lemma to hold, note that $d^2_{\mathbf{w}_n}(x,A)$ is always bounded by the integrable function $d^2_{\mathbbm{1}}(x,A)=d^2(x,A)$. The last equality follows from  Lemma \ref{continuity of distance}. Conversely, 
    \begin{equation} \label{liminf cont}
    \begin{aligned} 
        \int d_\mathbf{w}^2(x, A) d\mathbb{P}(x)&= \int \lim_{n \rightarrow \infty} d_{\mathbf{w}_n}^2(x, A)d\mathbb{P}(x)\\
        &\le (1+\epsilon) \int \liminf_{n \rightarrow \infty} d_{\mathbf{w}_n}^2(x,A_n) d\mathbb{P}(x) \\
        &\le (1+\epsilon) \liminf_{n \rightarrow \infty} \int d_{\mathbf{w}_n}^2(x, A_n) d\mathbb{P}(x),
    \end{aligned}
    \end{equation}
     where we used \eqref{peter paul3} for the first inequality and Fatou's lemma for the last inequality. 
    Since $\epsilon$ was arbitrary, we can get rid of it at \eqref{limsup cont} and \eqref{liminf cont}, and this completes the proof.
\end{proof}

\begin{reptheorem}{thm:3}
        Let $X$ be a random vector taking values in $\mathbb{R}^p$ that follows a uniform distribution on $\bigcup_{i=1}^2 B(a_i, \sqrt{r}/2)$, where $a_1 = (0, \ldots, 0)^t$ and $a_2 = (\overbrace{1, \ldots, 1}^{r}, 0, \ldots, 0)^t$. Then, $\mathbf{w} = (\overbrace{1, \ldots, 1}^{r}, 0, \ldots, 0)^t$ and $A = \{a_1, a_2\}$ is a stationary point to (\ref{population sparse kmeans}).
\end{reptheorem}
\begin{proof}
    First, fix $\mathbf{w} = (1, \ldots, 1, 0, \ldots, 0)^t$. \\
    Then the problem boils down to the $s$-dimensional problem as
    \[
        \max_{A' \subset \mathbb{R}^s} \quad \mathbb{E}\left[||X_1-\mu||^2 - \min_{a \in A'}||X_1-a||^2\right],
    \]
    where $X=(X_1^t, X_2^t)^t$ and $X_1$ is an $s$-dimensional random vector. Now the problem is equivalent to 
    \[
        \min_{A' \subset \mathbb{R}^s} \quad \mathbb{E}\left[\min_{a \in A'}||X_1-a||^2\right],
    \]
    which is a standard form arising in vector quantization \cite{graf:2007}.
    By Theorem 4.16 (Ball packing theorem) of \cite{graf:2007}, $A'=\{a_{11}, a_{21}\}$, where $a_1=(a_{11}^t, a_{12}^t)^t$ and $a_2=(a_{21}^t, a_{22}^t)^t$ is the optimal solution. This shows that $A=\{a_1, a_2\}$ is optimal to (\ref{population sparse kmeans}) holding $\bf w$ fixed. 
    
    Conversely, fix $A=\{a_1, a_2\}$. The objective function at (\ref{population sparse kmeans}) is expressed as
    \begin{equation*}
        \begin{aligned}
            &\quad \sum_{l=1}^p w_l Var(X_l) - \int_{\Omega_1} ||x-a_1||_{\mathbf{w}}^2 dP(x) - \int_{\Omega_1^c} ||x-a_2||_\mathbf{w}^2 dP(x) \\
            &= \sum_{l=1}^{s} w_l \{Var(X_l) - \int_{\Omega_1} (x_l -a_{1l})^2 dP(x) - \int_{\Omega_1^c} (x_l - a_{2l})^2 dP(x) \},
        \end{aligned}
    \end{equation*}
    where 
    \begin{equation*}
        \begin{aligned}
            \Omega_1 &= \{x \in \mathbb{R}^p : ||x-a_1||_\mathbf{w}^2 \leq ||x-a_2||_\mathbf{w}^2 \} \\
            &= \{x \in \mathbb{R}^p : \sum_{l=1}^i w_l(x_l-a_{1l})^2 \leq \sum_{l=1}^i w_l(x_l-a_{2l})^2 \} \\
            &= \{x \in \mathbb{R}^p : \sum_{l=1}^i w_l (2x_l - 1) \leq 0\}.
        \end{aligned}
    \end{equation*}    
    Since the objective function doesn't involve any term of $w_{i+1}, \cdots, w_p$, it can be inferred that the optimal solution entails $w_{i+1}=\cdots=w_p=0$. \\
    Moreover, the objective function is convex and symmetric. Let the objective function be denoted by $g({\bf w})$ holding $A$ fixed. Then,
    \begin{equation*}
        \begin{aligned}
            g(\lambda\mathbf{w_1} + (1-\lambda))\mathbf{w_2}) 
            &= \int \lambda ||x-a||_\mathbf{w_1}^2 + (1-\lambda) ||x-a||_{\mathbf{w_2}}^2 dP(x) \\
            &\hspace{1cm}- \int \min_{a \in A} \{ \lambda||x-a|| _{\mathbf{w_1}}^2 + (1-\lambda)||x-a||_\mathbf{w_2}^2 \} dP(x) \\
            &\leq \lambda g(\mathbf{w_1}) + (1-\lambda) g(\mathbf{w_2}),
        \end{aligned}
    \end{equation*}
    for all $0 < \lambda < 1$, and $g(\mathbf{w}) = g(P\mathbf{w})$ for every permutation matrix $P$. 
    Therefore, $g$ has a maximizer of the form $\alpha \mathbf{1}$ for $\alpha \geq 0$ (see Exercises 4.4 of \cite{boyd:2004}) 
    Since
    \[
    g(\alpha \mathbf{1}) = \alpha \int \bigl(||x-\mu||^2 - \min_{\theta \in \{\mu_1, \mu_2\}} ||x-\theta||^2 \bigr) dP(x) \geq 0 = g(\mathbf{0}),
    \]
    $\alpha>0$. This completes the proof.
\end{proof}

\subsection{Proofs of general (non-Euclidean) distance} 

\medskip 
\begin{reptheorem}{thm:4}
	Let $\hat{\theta}= (\hat{\bf{w}}, \{\hat{C_1}, \ldots, \hat{C_K}\})$ denote the minimizer of (\ref{orig sparse kmeans}) over $\mathcal{F} \times \Pi$, where $\mathcal{F} = \{ \mathbf{w} \in \mathbb{R}^p : ||\mathbf{w}||_2^2 \leq 1, ||\mathbf{w}||_1 \leq s, w_j \geq 0, \forall j \}$. Also denote by $\theta^* = (\mathbf{w}^*, \{C_1^*, \ldots, C_K^*\})$ the minimizer of corresponding population problem (\ref{general population sparse kmeans}) over $\mathcal{F} \times \Pi$. Assume ($\text{A1}^\prime$) and (A3). Then, with probability at least $1-4pt$, 
\begin{eqnarray}
	R(\hat{\theta}) - R(\theta^*) & \le&  2sM \sqrt{\frac{2}{n}\log(1/t)} + \frac{4sKM}{\delta^2} \left(2RC+ \sqrt{\frac{2}{n}\log(1/t)}\right)
	\nonumber\\
&& \qquad \qquad +  \frac{2sK}{\delta}\left( 2 \max_{1 \le j \le p}RC_j + M\sqrt{\frac{2}{n}\log(1/t)} \right) \nonumber
\end{eqnarray}
provided that $2RC+\sqrt{\frac{2}{n}\log{1/t}} \le \frac{\delta}{2}$, where
\[
	RC = \mathbb{E}\sup_{C \in \mathcal{P}, \mathcal{P} \in \Pi}\frac{1}{n} \left|\sum_{i=1}^n \epsilon_i \mathbb{I}(X_i \in C) \right|
\]
\[
	RC_j = \mathbb{E}\sup_{C \in \mathcal{P}, \mathcal{P} \in \Pi}\frac{1}{\lfloor n/2 \rfloor} \left| \sum_{i=1}^{\lfloor n/2 \rfloor} \epsilon_i d_j(X_i, X_{i+\lfloor n/2 \rfloor}) \mathbb{I}((X_i,X_{i+\lfloor n/2 \rfloor}) \in C^2) \right|
\]
\end{reptheorem}
\begin{proof}
As usual, we depend on the following risk bound
\[
	R(\hat{\theta})-R(\theta^*) \le 2 \sup_{\theta \in \mathcal{F} \times \Pi} |R_n(\theta)-R(\theta)|,
\]
where 
\[
	R_n(\theta) = \sum_{j=1}^p w_j \frac{1}{n-1} \left( \frac{1}{n}\sum_{i \ne i'} d_{i,i',j} - \sum_{k=1}^K \frac{1}{n_k} \sum_{i,i' \in C_k} d_{i,i',j} \right),
\]
 the properly scaled empirical risk. 
Our goal is to bound the supremum of an empirical process. First, note that for every $\theta = (\mathbf{w}, \mathcal{P})$, where $\mathbf{w} \in \mathcal{F}$ and $\mathcal{P} = \{C_1, \ldots, C_K\} \in \Pi$, 

\begin{equation*}
\begin{aligned}
|R_n(\theta) - R(\theta)| &\le \sum_{j=1}^p w_j \Bigg\{ \underbrace{\left| \frac{1}{n(n-1)} \sum_{i \ne i'}d_{i,i',j} - \mathbb{E}d_j(X_1,X_2) \right|}_{(1)_j} + \\
&\qquad \qquad \sum_{k=1}^K \underbrace{ \left| \frac{1}{n_k(n-1)} \sum_{i,i' \in C_k} d_{i,i',j} - \frac{1}{P(C_k)} \mathbb{E}\left[d_j(X_1, X_2)I\{(X_1, X_2) \in C_k^2\} \right] \right|}_{(2)_{j,k}} \Bigg\} \\
&\le \sum_{j=1}^p w_j \max_{1 \le j \le p} \left\{ (1)_j + \sum_{k=1}^K (2)_{j,k} \right\} \\
&\le s \max_{1 \le j \le p} \left\{ (1)_j + \sum_{k=1}^K (2)_{j,k} \right\}.
\end{aligned}
\end{equation*}

Now, our estimate is no longer dependent on $\mathbf{w}$. Therefore, taking supremum over possible $\theta$ gives
\begin{equation*}
\begin{aligned}
\sup_{\theta \in \mathcal{F} \times \Pi}|R_n(\theta) - R(\theta)| &\le s \max_{1 \le j \le p} \sup_{\mathcal{P} \in \Pi} \left\{ (1)_j + \sum_{k=1}^K (2)_{j,k} \right\} \\
&\le s \max_{1 \le j \le p} \left\{ (1)_j + \sum_{k=1}^K  \sup_{\mathcal{P} \in \Pi} (2)_{j,k} \right\} \\
&= s \max_{1 \le j \le p} \left\{ (1)_j + K  \sup_{C \in \mathcal{P}, \mathcal{P} \in \Pi} (2)_{j} \right\} \\
&\le s \max_{1 \le j \le p} (1)_j + sK \max_{1 \le j \le p} \sup_{C \in \mathcal{P}, \mathcal{P} \in \Pi} (2)_j,
\end{aligned}
\end{equation*}
 where the equality comes from the fact that $\sup_{\mathcal{P} \in \Pi} (2)_{j,k}$ is the same for every $k=1, \ldots, K$ as clustering is unaffected by the order of clusters. Here, we let 
\[
(2)_j = \left| \frac{1}{n_k(n-1)} \sum_{i,i' \in C} d_{i,i',j} - \frac{1}{P(C)} \mathbb{E}\left[d_j(X_1, X_2)I\{(X_1, X_2) \in C^2\} \right] \right|.
\]

The first part, $(1)_j$, which is simply the concentration of U-statistics can be handled by bounded difference inequality. With probability at least $1-2t$,
\begin{equation} \label{eqn:18}
	(1)_j \le M \sqrt{\frac{2}{n}\log(1/t)}
\end{equation}

The second part is further decomposed as
\begin{equation*}
\begin{aligned}
	\sup_{C \in \mathcal{P}, \mathcal{P} \in \Pi} (2)_j &\le \sup_{C\in \mathcal{P}, \mathcal{P} \in \Pi} \left | \frac{1}{P(C)} \right | \sup_{C \in \mathcal{P}, \mathcal{P} \in \Pi} \left| \frac{1}{n(n-1)}\sum_{i,i' \in C} d_{i,i',j} - \mathbb{E}\left[d_j(X_1, X_2)I\{(X_1, X_2) \in C^2\} \right] \right| \\
	& \phantom{{}\le \sup_{C\in \mathcal{P}, \mathcal{P} \in \Pi} \left | \frac{1}{P(C)} \right | \sup_{C \in \mathcal{P}, \mathcal{P} \in \Pi}} + \sup_{C \in \mathcal{P}, \mathcal{P} \in \Pi} \left| \frac{1}{n(n-1)} \sum_{i,i' \in C} d_{i,i',j} \right| \sup_{C \in \mathcal{P}, \mathcal{P} \in \Pi} \left| \frac{1}{n_k/n} - \frac{1}{P(C)} \right| \\
	& \le \frac{1}{\delta} \sup_{C \in \mathcal{P}, \mathcal{P} \in \Pi} \left| \frac{1}{n(n-1)}\sum_{i,i' \in C} d_{i,i',j} - \mathbb{E}\left[d_j(X_1, X_2)I\{(X_1, X_2) \in C^2\} \right] \right| \\
	& \phantom{{}\le \frac{1}{\delta} \sup_{C \in \mathcal{P}, \mathcal{P} \in \Pi} | \frac{1}{n(n-1)}\sum_{i,i' \in C} d_{i,i',j} - \mathbb{E}[d_j(X_1, X_2)I\{X}+ M \sup_{C \in \mathcal{P}, \mathcal{P} \in \Pi} \left| \frac{1}{n_k/n} - \frac{1}{P(C)} \right|,
\end{aligned}
\end{equation*}
 and we handle each two terms independently. For this, we first show a lemma useful for handling the second one. Here, $P_nf = \frac{1}{n}\sum_{i=1}^nf(X_i)$ and $Pf = \mathbb{E}[f(X_1)]$ following standard notations in empirical process theory. \\
\begin{lemma}\label{lem:5}
	Suppose $\sup_{f \in \mathcal{F}} \left| P_nf-Pf \right| \le \epsilon$ and $\sup_{f \in \mathcal{F}} Pf \ge \delta$. Then, $\sup_{f \in \mathcal{F}} \left| \frac{1}{P_nf} - \frac{1}{Pf} \right| \le \frac{2\epsilon}{\delta^2}$, provided that $\delta \ge 2\epsilon$.
\end{lemma}
\begin{proof}
	$\forall f \in \mathcal{F}$, 
	\begin{equation*}
	\begin{aligned}
		\left | \frac{1}{P_nf} - \frac{1}{Pf} \right| &= \frac{|P_nf-Pf|}{P_nf \cdot Pf} \\
		&\le \frac{\epsilon}{(Pf-\epsilon)Pf} \\
		&\le \frac{\epsilon}{(\delta-\epsilon)\delta} \le \frac{2\epsilon}{\delta^2}
	\end{aligned}
	\end{equation*}
\end{proof}
Note that by bounded difference inequality together with standard symmetrization argument (see Theorem 4.10 in \cite{wainwright:2019}), with probability at least $1-t$, 
\[
	\sup_{C \in \mathcal{P}, \mathcal{P} \in \Pi} \left| \frac{n_k}{n} - P(C) \right| \le 2RC + \sqrt{\frac{2}{n}\log{1/t}}
\]
Applying  Lemma \ref{lem:5} with $\epsilon$ equal to the RHS, it follows that with the same probability, 
\begin{equation}
	\sup_{C \in \mathcal{P}, \mathcal{P} \in \Pi} \left| \frac{1}{n_k/n} - \frac{1}{P(C)} \right| \le \frac{2}{\delta^2}\left( 2RC + \sqrt{\frac{2}{n}\log{1/t}} \right)
\end{equation}
since our assumption $2RC+\sqrt{\frac{2}{n}\log{1/t}} \le \frac{\delta}{2}$ guarantees the condition $\delta \ge 2\epsilon$ in Lemma \ref{lem:5}. \\
For the first term, we bound it using Lemma 6 from \cite{clemencon:2014}. With probability at least $1-t$, 
\[
	\sup_{C \in \mathcal{P}, \mathcal{P} \in \Pi} \left| \frac{1}{n(n-1)}\sum_{i,i' \in C} d_{i,i',j} - \mathbb{E}\left[d_j(X_1, X_2)I\{(X_1, X_2) \in C^2\} \right] \right| \le 2 RC_j + M \sqrt{\frac{2}{n}\log{1/t}}.
\]
Thus, with probability at least $1-2t$, 
\[
	\sup_{C \in \mathcal{P}, \mathcal{P} \in \Pi} (2)_j \le \frac{1}{\delta}\left(  2 RC_j + M \sqrt{\frac{2}{n}\log{1/t}} \right) + \frac{2M}{\delta^2} \left( 2RC + \sqrt{\frac{2}{n}\log{1/t}} \right).
\]

Therefore, with probability at least $1-2pt$, 
\[
	\max_{1 \le j \le p} (1)_j \le M \sqrt{\frac{2}{n}\log(1/t)}
\]
and with probability at least $1-2pt$,
\[
	\max_{1 \le j \le p} \sup_{C \in \mathcal{P}, \mathcal{P} \in \Pi} (2)_j \le \frac{1}{\delta}\left(  2 \max_{1 \le j \le p}RC_j + M \sqrt{\frac{2}{n}\log{1/t}} \right) + \frac{2M}{\delta^2} \left( 2RC + \sqrt{\frac{2}{n}\log{1/t}} \right)
\]
Putting these all together yields the theorem.
\end{proof}

\bibliographystyle{apalike}
\bibliography{ref}
\end{document}